\documentclass[11pt]{article}

\usepackage[utf8]{inputenc}
\usepackage[vietnamese,english]{babel}
\usepackage[all, cmtip]{xy}

\usepackage{stmaryrd}

\usepackage{amsmath}
\usepackage{amsxtra}
\usepackage{amscd}
\usepackage{amsthm}
\usepackage{amsfonts}
\usepackage{amssymb}
\usepackage{bbm}

\usepackage[bitstream-charter]{mathdesign}
\usepackage[T1]{fontenc}
\usepackage{fancyhdr}
\usepackage{tikz}
\usepackage{tikz-cd}
\usetikzlibrary{matrix,calc,arrows}

\numberwithin{equation}{section}

\newtheorem{theorem}{Theorem}

\newtheorem{lemma-construction}[theorem]{Lemma-Construction}
\newtheorem{proposition}[theorem]{Proposition}

\numberwithin{theorem}{section}
\theoremstyle{definition}

\theoremstyle{remark}

\newcommand\und{\underline}

\newcommand\indlim\varinjlim

\let\fg\undefined

\newcommand{\fg}{{\mathfrak g}}

\newcommand{\ft}{{\mathfrak t}}

\newcommand{\fC}{{\mathfrak C}}

\newcommand{\fS}{{\mathfrak S}}

\newcommand\ZZ{\mathbb{Z}}

\newcommand\bZ{{\ZZ}}

\newcommand\Z{\mathbb{Z}}

\newcommand\A{\mathbb{A}}
\newcommand\N{\mathbb{N}}

\newcommand\rN{\mathrm{N}}

\newcommand\rT{\mathrm{T}}

\newcommand\gl{\mathfrak{gl}}
\newcommand\GL{\mathrm{GL}}

\newcommand\Sp{\mathrm{Sp}}

\newcommand\tr{\mathrm{tr}}

\newcommand\id{\mathrm{id}}
\newcommand\Spec{\mathrm{Spec}}

\newcommand\Gm{\mathbb{G}_m}

\newcommand\Hom{\mathrm{Hom}}

\newcommand\End{\mathrm{End}}

\newcommand\Sym{\mathrm{S}}

\newcommand\diag{\mathrm{diag}}

\newcommand{\quash}[1]{}  
\newcommand{\nc}{\newcommand}

\nc{\al}{{\alpha}} \nc{\be}{{\beta}} \nc{\ga}{{\gamma}}
\nc{\ve}{{\varepsilon}} 
\nc{\La}{{\Lambda}}

\newcommand{\beqn}{\begin{equation*}}
\newcommand{\eeqn}{\end{equation*}}

\newcommand{\beq}{\begin{equation}}
\newcommand{\eeq}{\end{equation}}

\title{Invariant theory for the commuting scheme\\ of symplectic Lie algebras}
\author{Tsao-Hsien Chen and \foreignlanguage{vietnamese}{Ngô Bảo Châu}}

\date{}
\begin{document}
\maketitle

\begin{abstract}
We prove the Chevalley restriction theorem for the commuting scheme of symplectic Lie algebras.
The key step is the construction of the inverse map of the 
Chevalley restriction map called the spectral data map.
Along the way, we establish a certain multiplicative property of the Pfaffian
which is of independent interest.
\end{abstract}

\section{The Chevalley restriction theorem for commuting schemes}

Let $k$ be a field of characteristic zero, or of large prime characteristic. Let $G$ be a reductive group of $k$, $\mathfrak{g}$ its Lie algebra. For every $d\in \N$, we consider the commuting scheme $\fC^d_\fg$ consisting of elements $(x_1,\ldots,x_d)\in \mathfrak{g}^d$ such that $[x_i,x_j]=0$ for all $i,j\in \{1,\ldots,d\}$. The commuting scheme has always been of some interest in invariant theory but it was only recent that it appears as a primordial object in the study of moduli space of Higgs bundles for higher dimensional varieties. It is also poorly understood. For $d=2$, Richardson proved that the open subscheme of $\fC^{2,\rm rss}_{\fg}$ defined by the condition that $x_1,x_2$ are regular and semi-simple is dense smooth and irreducible. Although the reducedness of $\fC_\fg^2$ has the status of a folklore conjecture, there is a little evidence to expect $\fC^d_\fg$ to be reduced in general. For $d\geq 3$, $\fC^d_\fg$ is in general reducible, and components other than the closure of the regular semisimple open subscheme of $\fC^d_\fg$ are unlikely to be reduced. 

In this paper, we are more concerned with the categorical quotient 
\[\mathfrak C^d_{\mathfrak g}\sslash G=\Spec (k[\fC^d_\fg]^G)\] where $k[\fC^d_\fg]^G$ is the ring of $G$-invariant functions on $\fC^d_\fg$. 
We note that $G$ acts on $\mathfrak{g}$ by the adjoint action, and hence on $\mathfrak{g}^d$ 
by the diagonal adjoint action. This action leaves the commuting subscheme $\mathfrak C^d_{\mathfrak g}$ stable and we are interested in generalizing the Chevalley restriction theorem asserting a description of the categorical quotient $\mathfrak C^d_{\mathfrak g}\sslash G$ in terms of a Cartan subalgebra $\ft$ of $\fg$ equipped with the Weyl group action $W$. Let $T$ denote a maximal torus of $G$, $\ft$ its Lie algebra. The Weyl group $W=N_G(T)/T$ acts on $T$ and $\ft$. The embedding $\ft^d \to \mathfrak{g}^d$ factors through the commuting variety $\mathfrak C^d_\fg$ and it induces a homomorphism of algebras
\[ c:  k[\mathfrak C^d_\fg]^G \to k[\ft^d]^W \]
because the restriction of a $G$-invariant function to $\ft^d$ is obviously $W$-invariant. We conjecture that $c$ is always an isomorphism. The classical Chevalley restriction theorem corresponds to the case $d=1$.  
Note that, since $\ft^d\sslash W$ is known to be normal and reduced, the conjecture 
will imply  $\frak C^d_\fg \sslash G$ is normal and reduced.

For $G=\GL_n$, this conjecture was proved by Vaccarino in \cite{Vaccarino:2007wo} (not being aware of Vaccarino's result, we also reprove it in \cite{Chen:2020iu}). Vaccarino's proof in the case $G=\GL_n$ relies on the construction of a map in the opposite direction, to be called the spectral data map
\[s:k[\ft^d]^W \to k[\mathfrak C^d_\fg]^G \]
which is due to Deligne in the case $G=\GL_n$.  Once the spectral data map is constructed, to prove the Chevalley restriction theorem for commuting schemes it is enough to prove that $s\circ c$ and $c\circ s$ are identities. For this, one can use a a result of Procesi in \cite{Procesi:1976ep} which provides a system of generators of the ring $k[\fg^d]^G$. 

In this paper we will prove the Chevalley restriction theorem for the commuting scheme of the symplectic Lie algebra in following the same line on thought. The main novelty is the construction of the spectral data map for $G=\Sp_{2n}$. One should note that for general reductive groups, even in the case $d=1$, we don't know to construct this map without assuming first the Chevalley restriction theorem. For symplectic groups, a key ingredient in our construction is a certain multiplicative property of the Pfaffian. This is an elementary fact of linear algebra which seems to be new and of independent interest.

\section{Deligne's construction}

We will first recall Deligne's beautiful construction of the spectral data map for $\GL_n$
in  \cite[Section 6.3.1]{Deligne}. As a preparation, we will first recall Roby's concept of polynomial laws which will provide a convenient language for Deligne's construction (see \cite{Roby:1963gn} and also \cite{Ferrand}).

Let $A$ be a commutative ring. For every $A$-module $V$, we will denote $V_A$ the functor $R\to V\otimes_A R$ from the category of $A$-algebras to the category of sets. If $V$ and $N$ are $A$-modules, we will denote ${\rm P}(V,N)$ the set of morphism of functors $f:V_A \to N_A$. In case $N$ is not explicitly mentioned, we will understand that $N=A$, i.e., ${\rm P}(V)={\rm P}(V,A)$ and call ${\rm P}(V)$ the set of polynomial laws on $V$.

The connection between polynomial law and usual polynomial can be explained as follows. Let $S_A=A[X_1,\ldots,X_d]$ be the polynomial algebra with free variables $X_1,\ldots,X_d$. For every finite set of elements $v_1,\ldots,v_d \in V$, we have an element $X_1v_1+\cdots+X_d v_d \in V\otimes_A S_A$. If $f$ is a polynomial law on $V$, then $f_{\underline v}=f(X_1v_1 +\ldots X_d v_d)$ is an element of $S_A$ i.e. a polynomial of variables $X_1,\ldots,X_n$ with coefficients in $A$. If $V$ is a free $A$-module and $v_1,\ldots,v_d$ form a basis of $V$, then the polynomial $f_{\underline v}\in R$ determines the polynomial law $f$. 

A polynomial law $f$ on $V$ is said to be homogenous of degree $n$ if for every $A$-algebra $R$ and element $v\in V\otimes_A R$ we have $f(u m)=u^n f(m)$ for every $u\in R^\times$. If $V$ is a free $A$-module and $v_1,\ldots,v_d$ form a base of $V$, then $f$ is homogenous of degree $n$ if and only if  $f_{\underline v}$ is a homogenous polynomial of degree $n$. 

If $V$ is an $A$-module, we denote $\rT^n_A (V)$ the $n$th fold tensor power $V$ with itself over $A$ which is equipped with an action of the symmetric group $\fS_n$. We denote ${\rm TS}^n_A(V)$, the $n$th module of symmetric tensors of $V$ that is the submodule of $\rT^n_A (V)$ consisting of elements fixed under $\fS_n$, to be differentiated from ${\rm S}^n_A(V)$, the $n$th symmetric power of $V$ that is the largest quotient of $\rT^n_A (V)$ on which $\fS_n$ acts trivially. We have a map $V\to {\rm TS}^n V$ given by $v\mapsto v^{\otimes n}$. Roby proved that if $V$ is a free $A$-module, then there is a canonical bijection between the set of homogeneous polynomial laws $f$ of degree $n$ and the set of homogenous polynomial $h$ of degree $1$ on ${\rm TS}^e_A(V)$ characterized by the equality $f(v)=h(v^{\otimes n})$. We note that Roby states this theorem \cite[Theorem IV.1]{Roby:1963gn} with the divided power module instead of the symmetric tensors modules. These modules coincide however in the case where $V$ is a free $A$-module \cite[Propositions III.1, IV.5]{Roby:1963gn}.

If moreover, $V$ is an $A$-algebra, which is free as an $A$-module, and if $f$ is a multiplicative homogenous polynomial law of degree $n$ on $V$ i.e. if $f(xy)=f(x)f(y)$, then the corresponding degree 1 homogenous polynomial law on ${\rm TS}^n_A(V)$ is a homomorphism of algebras ${\rm TS}^e_A(V)\to A$ \cite[Proposition 2.5.1]{Ferrand}. 

We now consider the group $G=\GL_n$ whose Lie algebra is the space of matrices $M_n$ which is also equipped with a structure of algebras. Let $T_n$ be the diagonal torus of $\GL_n$ and $\ft_n$ its Lie algebra. The Weyl group $W$ is the symmetric group $\fS_n$ acting on $T_n$ and $\ft_n$ by permutation of coordinates. The commuting scheme of $\GL_n$ will be denoted by $\fC_n^d$. We will recall Deligne's construction of a $\GL_n$-invariant map
\[ s:  \mathfrak C^d_n \to \ft_n^d \sslash \fS_n\]
which roughly records the joint eigenvalues of set of commuting matrices. 

This map can be easily described at level of points in an algebraically closed field. If $x_1,\ldots,x_d\in \gl_n(k)$ are commuting matrices, we can equipped the $n$-dimensional vector space $V=k^d$ with a structure of $k[X_1,\ldots,X_d]$-module. Since $V$ is finite dimensional, it is supported by a finite subscheme of $\A^d=\Spec(k[X_1,\ldots,X_d])$
\[ V=\bigoplus_{\alpha\in k^n} V_\alpha \]
where $V_\alpha$ is annihilated by a power of the maximal ideal defining $\alpha$. We then set $s(\alpha)$ to be the 0-cycle
\[ s(V)=\sum_\alpha \dim(V_\alpha) \alpha\]
which is a $k$-point of $\ft_n^d \sslash \fS_n$. It is not clear how to generalize this cycle construction for commuting matrices with values in an arbitrary test ring $R$.

Now let $A$ be the coordinate ring of $\fC_n^d$ and  $\und x=(x_1,\ldots,x_d)\in \fC_n^d(A)$ the tautological $A$-point of $\fC_n^d$. Let $S=k[X_1,\ldots, X_d]$ be the polynomial algebra of variables $X_1,\dots,X_d$. The point $\und x$ gives rise to a homomorphism of algebras 
\[ p_{\und x}:S\otimes_k R\to \gl_n(A\otimes_k R) \]
with $p_{\und x}(X_i)=x_i$ for every $k$-algebra $R$. By composing with the determinant map we get a polynomial law on $S$
\[ f_{\und x}:S\otimes_k R\to A\otimes R\]
given by $f_{\und x}=\det\circ p_{\und x}$ which is homogenous of degree $n$ and multiplicative. 
By Roby's theorem, this is equivalent with a homomorphism of algebras
\[ h_{\und x}: {\rm TS}^n(S)\to A \]
which is a $R$-point of $\Spec({\rm TS}^n(S(V)))=\ft_n^d \sslash \fS_n$. Since $\det$ is $G$-invariant, $h_{\und x}$ is also $G$-invariant. As a result, we obtain the spectral data map 
\[s: {\rm TS}^n(S)\to A^G .\]
We note that this construction works without any restriction on the characteristic of the base field $k$.

\section{Multiplicative property of the Pfaffian}

Instead of the determinant, our construction of the spectral data map for symplectic groups relies on the Pfaffian function and its multiplicative property. The Pfaffian is a homogenous form of degree $n$ on the space of antisymmetric forms on $k^{2n}$. Since the Pfaffian is a square-root of the determinant one may ask the question whether it enjoys the same multiplicative property as the determinant. A priori the question is ill-posed for the product of antisymmetric matrices is not antisymmetric. We will show it is indeed possible to prove a multiplicative property of a function closely related the Pfaffian. This elementary result seems to be new and of independent interest. 

In this section, we assume that the base field $k$ is of characteristic zero or of odd prime characteristic.  Let $V$ be a $2n$-dimensional $k$-vector space. A bilinear form on $V$ is an element of the vector space $\Hom_k(V,V^*)$ which is equipped with an involution given by $x\mapsto x^*$. We observe that $\Hom_k(V,V^*)$ is canonically isomorphic to $V^*\otimes_k V^*$ equipped with the involution $v_1^*\otimes v_2^* \mapsto v_2^* \otimes v_1^*$. We have a decomposition into eigenspaces of this involution
\[ V^*\otimes_k V^* = \Sym^2 V^* \oplus \Lambda^2 V^* \]
where $\Sym^2 V^*$ corresponds tp the space of symmetric bilinear forms on $V$ and $\Lambda^2(V^*)$ the space of alternating bilinear forms on $V$. 

There is a canonical map
\[\Sym^n (\Lambda^2 V^*) \to \Lambda^{2n} V^* \]
given by the super-commutative multiplication law in the exterior algebra $\Lambda^\bullet V^*= \bigoplus_{i=0}^{2n} \Lambda^i V^*$. It follows that we have a degree $n$ homogenous polynomial law \[\mu:\Lambda^2 V^* \to \Lambda^{2n} V^*\]
which associates to an alternating form $\omega\in \Lambda^2 V^*$ the image of $\omega^n \in \Sym^n(\Lambda^2 V^*)$ in the line $\Lambda^{2n} V^*$.  We note that $\omega\in \Lambda^2 V^*$ is a non-degenerate alternating bilinear form if and only if $\mu(\omega)$ is a non-zero vector of $\Lambda^{2n} V^*$. If it is the case, we will say that $\omega$ is a symplectic form. 

We also note that $\mu$ is a square-root of the determinant in the following sense. A bilinear form $b$ on $V$ induces a bilinear form on the determinant line $\Lambda^{2n} V$
\[\det(b)\in \Hom_k(\Lambda^{2n} V, \Lambda^{2n} V^* )= \Lambda^{2n} V^* \otimes \Lambda^{2n} V^* .\]
Then for every alternating form $\omega\in\Lambda^2 V^*$ we have the identity
\begin{equation} \label{det-square-root}
	\det(\omega)=\mu(\omega)\otimes \mu(\omega).
\end{equation}

We will now fix a symplectic form $\omega_0\in \Lambda^2 V^*$ and consider the symplectic group $G=\Sp_{2n}$ of all linear transformations of $V$ that preserve $\omega_0$. The Lie algebra $\fg$ of $G$ is the subspace of $\gl(V)$ of matrices $x\in \gl(V)$ such that 
\[ \omega_0(xu,v)=-\omega_0(u,xv)\]
for all vectors $u,v \in V$. This is equivalent to say that the identity 
$\omega_0 x=-x^* \omega_0$ holds in $\Hom_k(V,V^*)$. In this case we have $(\omega_0 x)^*=-x^* \omega_0^*= x^* \omega_0=\omega_0 x$ and therefore $\omega_0 x\in \Sym^2 V^*$. In other words, the map $x\mapsto \omega_0 x$ induces an isomorphism of $k$-vector spaces $\fg\to \Sym^2 V^*$. 

We will also consider the decomposition $\gl(V)=\fg \oplus \fg^+$ where $\fg^+$ is the subspace of matrices $x\in \End(V)$ such that 
\[ \omega_0(xu,v)=\omega_0(u,xv)\]
for all vectors $u,v \in V$. The map $x\mapsto \omega_0 x$ induces an isomorphism of $k$-vector spaces $\fg^+\to \Lambda^2 V^*$. We will define a Pfaffian norm $\rN_+:\fg^+\to \A^1$ by the equality
\[ \rN_+(x) \mu(\omega_0)=\mu(\omega_0 x)\]
for every $k$-algebra $R$ and $x\in \fg^+(R)$. Note that since $\mu(\omega_0)$ is a generator of the free $R$-module $\Lambda^{2n} V^* \otimes_k R$ the above equality defines $\rN_+(x)\in R$ uniquely.

Applying the identity \eqref{det-square-root} to $x\omega_0 \in \Lambda^2 V^*$, we get the equality \[\det(\omega_0 x)= \mu(\omega_0 x)\otimes \mu(\omega_0 x)\] in the line $(\Lambda^{2n} V^*)^{\otimes 2}$. It follows that 
\begin{equation}\label{Norm-square-root}
	\det(x)=\rN_+(x)^2
\end{equation}
for all $x\in \fg^+$. Note that since this equality is valid for all $x\in \fg^+\otimes_k R$ for all $k$-algebra $R$, $\det=\rN_+^2$ can be seen as an equality in the coordinate ring of $\fg^+$. 

As a square-root of the determinant, we may expect that the function $\rN_+$ satisfies a multiplicative property as the determinant does. However, the multiplicativity does not make sense a priori as the subspace $\fg^+$ of $\gl(V)$ is not stable under matrix multiplication. We note that for $x,y\in \fg^+$, $xy\in \fg^+$ if and only if $xy=yx$. 
The multiplicicativity of the Pfaffian norm $\rN_+$ then makes sense as an identity in the coordinate ring of the commuting subscheme $\fC_{\fg^+}$ of $\fg^+\times \fg^+$.

\begin{proposition} \label{weak-multiplicativity}
	The equality $\rN_+(xy)=\rN_+(x)\rN_+(y)$ holds in the coordinate ring of $\fC_{\fg^+}$.
\end{proposition}

\begin{proof}
	This is equivalent to prove that for every point $(x,y)\in \fC_{\fg^+}(R)$ with values in an arbitrary $k$-algebra $R$, the identity $\rN_+(xy)=\rN_+(x)\rN_+(y)$ holds in $R$. For this we introduce new formal variables $\alpha,\beta$ and consider commuting elements
	\[ 1+\alpha x, 1+\beta y \in \fg^+(R[\alpha,\beta])\]
	with values in the polynomial ring $R[\alpha,\beta]$. We also have
	\[(1+\alpha x)(1+\beta y)\in \fg^+(R[\alpha,\beta]) \]
	because $xy=yx$. We now have elements $P_x=\rN_+(1+\alpha x), P_y=\rN_+(1+\beta y)$ and $P_{xy}=\rN_+((1+\alpha x)(1+\beta y))$ in $R[\alpha,\beta]$ which are polynomial with free coefficients equal to $1$. We also note that $P_x\in R[\alpha]$ is a polynomial of degree at most $n$ in $\alpha$ whose coefficient of $\alpha^n$ is $\rN_+(x)$, $P_y\in R[\beta]$ is a polynomial of degree at most $n$ in $\beta$ whose coefficient of $\beta^n$ is $\rN_+(y)$, and $P_{xy}=\rN_+((1+\alpha x)(1+\beta y))\in R[\alpha,\beta]$ is a polynomial of degree at most $n$ in both variables $\alpha$ and $\beta$ whose coefficient of $\alpha^n \beta^n$ is $\rN_+(xy)$. To prove $\rN_+(xy)=\rN_+(x)\rN_+(y)$, it is enough to prove the equality of polynomials $P_{xy}=P_x P_y$. 
	
	Since the ring of polynomials $R[\alpha,\beta]$ embeds in the ring of formal series $R[[\alpha,\beta]]$, it is enough to prove the equality $P_{xy}=P_x P_y$ in $R[[\alpha,\beta]]$. We note that $P_x, P_y, P_{xy}$ are now invertible elements of $R[[\alpha,\beta]]$ which is a limit of thickening of $R$. Using the fact that the square map $\Gm\to \Gm$ is étale (here we use the assumption $\rm{char}(k)>2$), for every $g\in R[[\alpha,\beta]]^\times$ with free coefficient $g_0\in R^\times$, and for every square-root $f_0\in R^\times$ of $g_0$, there exists a unique $f\in R[[\alpha,\beta]]^\times$ with free coefficient $f_0$ such that $f^2=g$.

	Using the equality \eqref{Norm-square-root} we have $P_{xy}^2= P_x^2 P_y^2$. The fact that $P_x,P_y,P_{xy}$ have free coefficients 1 implies now the equality $P_{xy}=P_x P_y$ as formal series, and thus as polynomials.
\end{proof}

We observe that it is possible to prove the equality $\rN_+(xy)=\rN_+(x)\rN_+(y)$ holds for every point $\fC_{\fg^+}$ with value in a field using their simultaneous triangulation. This implies that the equality $\rN_+(xy)=\rN_+(x)\rN_+(y)$ holds in the reduced quotient of $k[\fC_{\fg^+}]$. However we do not know whether $k[\fC_{\fg^+}]$ is reduced.

\section{Spectral data map for symplectic groups}

Let $V$ be a $2n$-dimensional $k$-vector space equipped with a symplectic form $\omega_0\in \Lambda^2 V^*$. The group $G=\Sp_{2n}$ is the subgroup of $\GL_{2n}$ preserving $\omega_2$. The Lie algebra $\fg=\mathfrak{sp}_{2n}$ consists of elements $x\in \gl(V)$ such $x^* \omega_0=-\omega_0 x$. We have an orthogonal complement $\fg^+$ of $\fg$ in $\gl_{2n}$ consisting of $x\in\gl(V)$ such that $x^* \omega_0=\omega_0 x$. We observe that if matrices $x,y\in \fg$ are commuting matrices then we have $xy\in \fg^+$. We have also noted that for commuting elements $x,y\in \fg^+$ we have $xy\in \fg^+$. If $x\in \fg$ and $y\in \fg^+$ are commuting matrices then we have  $xy \in \fg$.

Let $A$ denote the coordinate ring $k[\fC_\fg^d]$ of the commuting scheme and $(x_1,\ldots,x_d)\in \fg(A)^d$ the universal sequence of commuting matrices. 
Let $S=k[X_1,\ldots,X_d]$ be the polynomial ring with $d$ variables. 
The commutation property implies that for every $k$-algebra $R$ there exists a morphism of rings 
\[ p: S\otimes_k R \to \mathfrak{gl}_{2n}(A\otimes_k R) \]
sending $X_i\mapsto x_i$. For every $\und a=(a_1,\ldots,a_d)\in \Z_{\geq 0}^d$ the image $p(X_1^{a_1} \ldots X_d^{a_d})$ lies in $\fg(R)$ or $\fg^+(R)$ depending on whether $a_1+\cdots+a_d$ is odd or even. If $S^+$ denotes the subalgebra of $S$ generated by the monomials $X_1^{a_1} \ldots X_d^{a_d}$ with $a_1+\cdots+a_d$ even then we have $p(S^+)\in \fg^+(R)$. Let us denote $p^+:S^+\otimes_k R\to \fg^+(A\otimes_k R)$ the restriction of $p$ to $S^+$. Composing with the Pfaffian norm $\rN_+:\fg^+\to \A^1$, we have a polynomial law 
\[ \rN_+:S^+\otimes_k R \to A\otimes_k R \]
which is homogenous of degree $n$ and multiplicative according to Proposition \ref{weak-multiplicativity}. By Roby's theorem, this gives rise to a morphism of $k$-algebras 
\[s:((S^+)^{\otimes n})^{\fS_n} \to A \]
satisfying $s(q^{\otimes n})= \rN_+(p^+(q))$
for all $q\in S^+$. Since $\rN_+$ is $G$-invariant, the induced map $s$ is also $G$-invariant, and as a result, the image of $s$ is contained in $A^G$.

Let $\tau$ denote the involution of $S$ given by $\tau(X_i)=-X_i$. 
We note that $S^+$ is the subalgebra of $S$ of fixed points of $\tau$. As a result, $((S^+)^{\otimes n})^{\fS_n}$ can be identified with the subalgebra of $S^{\otimes n}$ of fixed points under $(\Z/2\Z)^n \rtimes \fS_n$. This will permit us to identify $((S^+)^{\otimes n})^{\fS_n}$ with $k[\ft^d]^W$ via some explicit choice of the Cartan algebra. This choice will be ultimately irrelevant as we will prove that $s$ is an inverse to the map $c$ that doesn't depend on the choice of the Cartan algebra.

Let $V$ be now the standard $2n$-dimensional vector space $k^{2n}$ with basis $e_1,\ldots,e_{2n}$ and $\omega_0$ the standard symplectic form given by 
	\[ \omega_0(e_i,e_j)=\begin{cases}
		0 & \mbox{if } i+j \neq 2n+1\\
		1 & \mbox{if } i+j=2n+1, i\leq n \\
		-1 & \mbox{if } i+j=2n+1, i\geq n+1 \\
	\end{cases} \]
The subspace $\ft$ of diagonal matrices of the form $\diag(b_1,\ldots,b_n,-b_n,\ldots,-b_1)$ will then be a Cartan algebra of $\fg$ equipped with the obvious action of $W=\{ \pm 1\}^n \rtimes \fS_n$. The coordinate ring of $\ft$ is then the polynomial ring $k[b_1,\ldots,b_n]$ where the $b_i$ are the coordiantes given by entries of the diagonal matrix as above. Let $B=k[\ft^d]$ denote the coordinate ring of $\ft^d$ and $(y_1,\ldots,y_d)\in \ft^d(B)$ the tautological $B$-point of $\ft^d$. We consider the elements $b_j(y_i)\in B$ with $1\leq i \leq d$ and $1\leq j\leq n$ and the isomorphism of algebras $\beta:S^{\otimes n}\to B$ given by 
\begin{equation} \label{beta}
	\beta(X_{j,i})=b_j(y_i)
\end{equation}
where $X_{j,1},\ldots,X_{j,d}$ are the coordinates of the $j$th copy of $S$. By restriction  we have an isomorphism of algebras
\[ \beta:((S^+)^{\otimes n})^{\fS_n}= (S^{\otimes n})^{(\Z/2\Z)^n \rtimes \fS_n} \to k[\ft^d]^W.\] 
It follows that we have a morphism of algebras
\begin{equation} \label{spectral-data-map}
	s:k[\ft^d]^W\to k[\fC_\fg^d]^G
\end{equation} 
such that 
\begin{equation} \label{d-N}
	s(\beta(q^{\otimes n}))= \rN_+(p^+(q))
\end{equation}
for all $q\in\ S^+$. It remains to prove that the spectral data map $s$ is inverse to the Chevalley restriction.

\begin{theorem}\label{main}
The map $c:k[\fC_{\fg}^d]^G\to k[\ft^d]^W$
is an isomorphism with the inverse given by the spectral data map
$s:k[\ft^d]^W\to k[\fC_{\fg}^d]^G$ of \eqref{spectral-data-map}.
\end{theorem}

We need to show that the compositions $c\circ s$ and $s\circ c$ are equal to the identities. To this end, we introduce a set of generators for the rings $k[\ft^d]^W$ and $k[\frak C_{\frak g}^d]^G$ respectively, and then we check the desired property on those generators. Following Procesi as in \cite{Procesi:1976ep} these functions are constructed as certain traces. We only use the assumption that $k$ is of characteristic zero or of prime characteristic large enough to prove that these trace functions form a system of generators. We record this fact to facilitate the work of those who may want make the assumption on the characteristic explicit.

For every $\und a=(a_1,\ldots,a_d)\in \bZ_{\geq 0}^d$, we define the element $\phi_{\und a}\in A$ given by 
\begin{equation}
\phi_{\und a}=\tr(x_1^{a_1}\cdots x_d^{a_d}) 
\end{equation}
where $(x_1,\ldots,x_d)\in \fC_\fg^d(A)$ is the universal point of the commuting scheme. Since the trace is $G$-invariant, we have $\phi_{\und a}\in A^G$. We note that if $a_1+\cdots+a_d$ is odd then $x_1^{a_1}\cdots x_d^{a_d}\in \fg(A)$ and if $a_1+\cdots+a_d$ is even then $x_1^{a_1}\cdots x_d^{a_d}\in \fg^+(A)$. It follows that $\phi_{\und a}=0$ if $a_1+\cdots+a_d$ is odd. For this reason we will only consider the functions $\phi_{\und a}$ with $a_1+\cdots+a_d$ even.

\begin{proposition} The functions $\phi_{\und a}$ with $\und a\in \bZ^d_{\geq 0}$ form a set of generators of $k[\fC_{\fg}^d]^G$. 
\end{proposition}

\begin{proof}
The first statement is consequence of a result of Procesi \cite[Theorem 10.1]{Procesi:1976ep}. Procesi describe a set of functions on $\gl(V)^d$ invariant under the diagonal action of the symplectic group $G$ which generate the ring $k[\gl(V)^d]^G$. Since $\fC_{\fg}^d$ is clearly a closed subscheme of $\gl(V)^d$,, the restriction map $k[\gl(V)^d]^G\to k[\fC_{\fg}^d]^G$ is surjective (this is true in characteristic zero because $G$ is linearly reductive, and therefore it is also true for prime characteristic large enough). It follows that the restriction of Procesi's functions form a set of generators of $k[\fC_{\fg}^d]^G$. Using the fact that $x_1,\ldots,x_d$ are commuting elements of the symplectic Lie algebra, it is easy to see that Procesi's functions restrict to our functions $\phi_{\und a}$. 
\end{proof}

\begin{proposition}
	If $\psi_{\und a}=c(\phi_{\und a})$ then the functions $\psi_{\und a}$ generate $k[\ft^d]^W$ and we have 
	\[s(\psi_{\und a})=\phi_{\und a}.\] 
\end{proposition}

\begin{proof}
Let $B=k[\ft^d]$ denote the coordinate ring of $\ft^d$ and $(y_1,\ldots,y_d)\in \ft^d(B)$ the tautological $B$-point of $\ft^d$. We have observed that $B$ is a polynomial algebras of the variables $b_j(y_i)$ for $1\leq i\leq d$ and $1\leq j\leq n$.
By computing the trace of the matrix $y_1^{a_1}\cdots y_d^{a_d}$ we get 
\begin{equation}
	\psi_{\und a}=
	\begin{cases} 
	0 & \mbox{if } a_1+\cdots+a_d\mbox{ is odd,}\\
	2 \sum_{j=1}^n \prod_{i=1}^d b_j(y_i)^{a_i} & \mbox{if } a_1+\cdots+a_d\mbox{ is even,}
	\end{cases}
\end{equation} 
We derive from these explicit formulas that the functions $\psi_{\und a}$ for $\und a\in \Z_{\geq 0}^d$ with $a_1+\cdots+a_d$ even generate $B^W$. Here we need to assume the characteristic of $k$ is either zero or no less than $2n$ in order to perform usual manipulation with symmetric polynomials. 

In order to prove the equality $s(\psi_{\und a})=\phi_{\und a}$ for $\und a\in \Z_{\geq 0}^d$ with $a_1+\cdots+a_d$ even, we will calculate the determinant of the element \[p^+(\theta_{\und a})=t \id - x_1^{a_1}\cdots x_d^{a_d}\in \fg^+(A\otimes_k R)\] 
where $\theta_{\und a}= t-X_1^{a}\cdots X_d^{a_d} \in R\otimes_k S^+$ with the test ring $R=k[t]$ being the algebra of polynomials in one variable $t$. 

On the one hand, with the usual formula for the characteristic polynomial we have
\[ \det(p^+(\theta_{\und a}))=t^{2n}-\phi_{\und a}\, t^{2n-1} +\mbox{terms of lower degrees in }t
  \]
On the other hand after \eqref{Norm-square-root} we have $\det(p^+(\theta_{\und a}))=\rN_+(\theta_{\und a})^2$
where $\rN_+(\theta_{\und a})= s(\beta(\theta_{\und a}^{\otimes n}))$ by \eqref{d-N}. Using the identification \eqref{beta}, the image of the element $\theta_{\und a}^{\otimes n}\in R\otimes_k S^{\otimes n}$ in $R\otimes B$ is given by the formula
\begin{eqnarray*}
	\beta(\theta_{\und a}^{\otimes n})& = & \prod_{j=1}^n (t-\prod_{i=1}^d b_j(y_i)^{a_i})\\ 
	& = & t^n - \sum_{j=1}^n \prod_{i=1}^d b_j(y_i)^{a_i} t^{n-1} + \mbox{terms of lower degrees in }t
\end{eqnarray*}
It follows that
\[ \det(p^+(\theta_{\und a}))= t^{2n} - s(\psi_{\und a})\, t^{2n-1} + \mbox{terms of lower degrees in }t \]
Comparing the coefficient of $t^{2n-1}$ in the two polynomial expression of $\det(p^+(\theta_{\und a}))\in A[t]$ we derive the desired equality $s(\psi_{\und a})=\phi_{\und a}$. 
 \end{proof}
 
To finish the proof of  Theorem \ref{main} we observe that the 
propositions above imply that the
 compositions $c\circ s$ and $c\circ s$ are equal to the identities
 on the generators $\psi_{\und a}$ and $\phi_{\und a}$ respectively.

\section*{Acknowledgement}

The research of Tsao-Hsien Chen is supported by 
NSF Grant DMS-2001257 and the S. S. Chern Foundation. The research of \foreignlanguage{vietnamese}{Ngô Bảo Châu} is supported by NSF grant DMS-1702380 and the Simons foundation.


{}
\end{document}